\documentclass[11pt,a4paper]{article}
\usepackage{pdfsync} 

\usepackage{amsfonts,amsmath,amsthm}
\newtheorem{theorem}{Theorem}
\numberwithin{theorem}{section}
\newtheorem{lemma}{Lemma}
\newtheorem{corollary}{Corollary}

\usepackage{amssymb}
\usepackage{hyperref}
\usepackage{graphicx}
\usepackage{wrapfig}
\usepackage{subfig}
\usepackage{epsfig}
\usepackage{float,mathrsfs}
\usepackage{enumitem}

\usepackage{bbm}

\usepackage{algpseudocode}
\usepackage[font=scriptsize]{caption}
\usepackage{chngcntr} 
\counterwithout{equation}{section}
\counterwithout{figure}{section}
\counterwithout{theorem}{section}

\usepackage{color}

\definecolor{refkey}{rgb}{0.9451,0.2706,0.4941}\definecolor{labelkey}{rgb}{0.9451,0.2706,0.4941}

\definecolor{darkred}{RGB}{139,0,0}
\definecolor{darkgreen}{RGB}{0,100,0}
\definecolor{darkmagenta}{RGB}{139,0,139}

\newcommand{\bsx}{{\boldsymbol{x}}}
\newcommand{\bsm}{{\boldsymbol{m}}}
\newcommand{\bsmu}{{\boldsymbol{\mu}}}
\newcommand{\bsy}{{\boldsymbol{y}}}

\newcommand{\bbN}{{\mathbb{N}}}

\newcommand{\bsnu}{{\boldsymbol{\nu}}}

\newcommand{\bsb}{{\boldsymbol{b}}}

\newcommand{\bsxi}{{\boldsymbol{\xi}}}

\newcommand{\calF}{\mathcal{F}}

\title{Application of dimension truncation error analysis to high-dimensional function approximation in\\uncertainty quantification}

\author{Philipp A.~Guth\footnotemark[2]\and Vesa Kaarnioja\footnotemark[3]}

\renewcommand{\thefootnote}{\fnsymbol{footnote}}

\begin{document}
\maketitle

\begin{abstract}
Parametric mathematical models such as parameterizations of partial differential equations with random coefficients have received a lot of attention within the field of uncertainty quantification. The model uncertainties are often represented via a series expansion in terms of the parametric variables. In practice, this series expansion needs to be truncated to a finite number of terms, introducing a dimension truncation error to the numerical simulation of a parametric mathematical model. There have been several studies of the dimension truncation error corresponding to different models of the input random field in recent years, but many of these analyses have been carried out within the context of numerical integration. In this paper, we study the $L^2$ dimension truncation error of the parametric model problem. Estimates of this kind arise in the assessment of the dimension truncation error for function approximation in high dimensions. In addition, we show that the dimension truncation error rate is invariant with respect to certain transformations of the parametric variables. Numerical results are presented which showcase the sharpness of the theoretical results.
\end{abstract}

\footnotetext[2]{Johann Radon Institute for Computational and Applied Mathematics, Austrian Academy of Sciences, Altenbergerstra{\ss}e 69, A-4040 Linz, Austria, {\tt philipp.guth@ricam.oeaw.ac.at}}
\footnotetext[3]{Department of Mathematics and Computer Science, Free University of Berlin, Arnimallee 6, 14195 Berlin, Germany, {\tt vesa.kaarnioja@fu-berlin.de}}

\renewcommand{\thefootnote}{\arabic{footnote}}

\section{Introduction}

In the field of uncertainty quantification it is common to study mathematical models with uncertain influences parameterized by countably infinite sequences of random variables. Consider, for instance, an abstract model $M\!:X\times U\to Y$ such that
\begin{equation}\label{eq:GuKa_model}
M(g(\bsy),\bsy)=0,
\end{equation}
where $X$ and $Y$ are separable Hilbert spaces and $U$ is a nonempty subset of the infinite-dimensional sequence space of parameters $\mathbb R^{\mathbb N}$. The solution $g(\bsy)\in X$ to~\eqref{eq:GuKa_model} for $\bsy\in U$, if it exists, may be computationally expensive to evaluate. To this end, it may be preferable to instead \emph{approximate} $g$ using a surrogate which is cheap to evaluate and hence enables, e.g., efficient sampling of the (approximated) solution.

Some possible surrogate models include, but are not limited to, Gaussian process regression~\cite{GPR1}, reduced basis approaches~\cite{ROM2,ROM1}, generalized polynomial chaos expansions \cite{GPC1,GPC2}, neural network approximations~\cite{NN4,NN2,NN3,NN1}, and kernel interpolation based on lattice point sets~\cite{kkkns22,Kernel1,Kernel2}. The results presented in this paper are particularly well-suited to the analysis of kernel methods used in conjunction with the so-called \emph{periodic model} discussed in~\cite{hhkks22,kkkns22,kks20}, and we will devote a section of this paper to explore the application of our dimension truncation results to the periodic model of uncertainty quantification. 

The construction of a numerical surrogate is often based on collocating the target function over a cubature point set, such as Monte Carlo or quasi-Monte Carlo nodes. Therefore a natural first step for the numerical treatment of~\eqref{eq:GuKa_model} is the approximation by a dimensionally-truncated model $M_s\!:X\times U_s\to Y$ such that
\begin{equation*}
M_s(g_s(\bsy_{\leq s}),\bsy_{\leq s})=0,%
\end{equation*}
where $\varnothing\neq U_s\subseteq \mathbb R^s$ and $g_s(\bsy_{\leq s})\in X$ for all $\bsy_{\leq s}\in U_s$.  Consider the problem of finding a surrogate solution $g_{s,n}:=A_n(g_s)$ using an algorithm $A_n$ which uses $n$ point evaluations of the $s$-dimensional function $g_s$, where the surrogate belongs to $X$ such that
$$
\|g_s-g_{s,n}\|_{L_{\bsmu}^2(U;X)} \xrightarrow{n\to\infty}0
$$
with some known convergence rate and ${\bsmu}$ indicating a probability measure on $U$. The total error of the approximation obtained in this fashion can be estimated using the triangle inequality
$$
\|g-g_{s,n}\|_{L_{\bsmu}^2(U;X)}\leq \|g-g_s\|_{L_{\bsmu}^2(U;X)}+\|g_s-g_{s,n}\|_{L_{\bsmu}^2(U;X)}.
$$
In this paper we focus on the first term---the \emph{dimension truncation error}---which is independent of the chosen approximation scheme $A_n$.

Dimension truncation error rates are typically studied for problems involving partial differential equations (PDEs) with random inputs. For integration problems a dimension truncation rate is derived in~\cite{kss12} for the source problem with an affine parameterization of the diffusion coefficient. This rate was then improved by \cite{gantner} in the generalized context of affine parametric operator equations. Dimension truncation has also been studied for coupled PDE systems arising in optimal control problems under uncertainty~\cite{guth21}, in the context of the periodic model of uncertainty quantification for both numerical integration~\cite{kks20} and kernel interpolation~\cite{kkkns22}, as well as for Bayesian inverse problems governed by PDEs~\cite{gantner2,hks21}. The results in these papers have been proved using Neumann series, which is known to work well in the affine parametric setting, but may lead to suboptimal results if the problem depends nonlinearly on the parameters.

The use of Taylor series in the assessment of dimension truncation error rates has previously been considered by~\cite{matern1} within the context of elliptic PDEs equipped with lognormal random input fields. In the non-affine setting, using Taylor series makes it possible to derive dimension truncation error rates by exploiting the parametric regularity of the problem, whereas the Neumann series approach relies fundamentally on the affine parametric structure of the model. The Taylor series approach has been applied in \cite{GGKSS2019}, and motivated the authors in \cite{guth22} and \cite{gk22} to derive dimension truncation error rates for sufficiently smooth, Banach space valued integrands, and with parameters following a generalized $\beta$-Gaussian distribution. An overview of the various dimension truncation error bounds studied in the literature is given in Table~\ref{table:glossary}.

\begin{table}[!t]
\begin{center}
\begin{tabular}{l|l|l}
 &Integration&Function approximation\\
\hline Affine parametric&\cite{gantner,kss12}  &\cite{kkkns22}\\
operator equation setting&rate $\mathcal O(s^{-\frac{2}{p}+1})$&rate $\mathcal O(s^{-\frac{1}{p}+\frac12})$\\
\hline Non-affine parametric&\cite{GGKSS2019,gk22}&this paper\\
operator equation setting&rate $\mathcal O(s^{-\frac{2}{p}+1})$&rate $\mathcal O(s^{-\frac{1}{p}+\frac12})$
\end{tabular}
\end{center}
\caption{An overview of various dimension truncation results.}\label{table:glossary}
\end{table}

This paper is structured as follows. Subsection 1.1 introduces the multi-index notation used throughout the paper. The problem setting is introduced in Section 2, including the central assumptions for the ensuing dimension truncation analysis. Section 3 contains
the $L^2$ dimension truncation theorem for Hilbert space valued functions, and in Section 4 we discuss the invariance of the dimension truncation rate under certain transformations of the variables. Numerical experiments assessing the sharpness of our theoretical results are presented in Section 5. The paper ends with some conclusions in Section~6.

\subsection{Notations and preliminaries}
Throughout this paper, boldfaced symbols are used to denote multi-indices while the subscript notation $m_j$ is used to refer to the $j$-th component of multi-index $\bsm$. 
Let
$$
\mathcal F:=\{\bsm\in \mathbb N_0^{\mathbb N}:|\bsm|<\infty\}
$$
denote the set of finitely supported multi-indices, where the order of multi-index $\bsm$ is defined as
$$
|\bsm|:=\sum_{j\geq 1}m_j.
$$
Moreover, we denote
$$
|\bsm|_\infty := \max_{j\geq 1}m_j,
$$
and, for any sequence $\bsx := (x_j)_{j =1}^{\infty}$ of real numbers and $\bsm \in \calF$, we define
$$
\bsx^{\bsm} := \prod_{j\geq 1} x_j^{m_j},
$$
where we use the convention $0^0 := 1$.

\section{Problem setting}

Let $X$ be a real separable Hilbert space, $U:=[-\tfrac12,\tfrac12]^{\mathbb N}$ a set of parameters, and suppose that $g(\bsy)\in X$ is a parameterized family of functions with smooth dependence on $\bsy\in U$. We define $g_s(\bsy) := g(\bsy_{\leq s},\boldsymbol{0}) := g(y_1,\ldots,y_s,0,0,\ldots)$ and assume that $\bsmu(\mathrm d\bsy) := \bigotimes_{j\geq1}\mu(\mathrm dy_j)$ is a countable product probability measure, i.e., $\bsmu(U) = 1$.
We suppose that
\begin{enumerate}
\item \label{assumpA1}For $\bsmu$-a.e.~$\bsy\in U$, there holds
$$
\|g(\bsy)-g_s(\bsy)\|_X\xrightarrow{s\to\infty}0.
$$
\item \label{assumpA2}Let $(\Theta_k)_{k\geq 0}$ and $\bsb:=(b_j)_{j\ge 1}$ be sequences of nonnegative numbers such that $\bsb\in \ell^p(\mathbb N)$ for some $p\in (0,1)$ and $b_1\geq b_2\geq\cdots$. Suppose that $g$ is continuously differentiable up to order $k+1$, with
$$
\|\partial^{\bsnu}g(\bsy)\|_X\leq \Theta_{|\bsnu|}\bsb^{\bsnu}
$$
for all $\bsy \in U$ and for all $\bsnu\in\mathcal F_k := \{\bsnu \in \bbN_0^{\bbN}:|\bsnu|\leq k+1\}$, where $k:=\lceil \frac{1}{1-p}\rceil$.
\item \label{assumpA3}There holds $\int_{-1/2}^{1/2}y_j\,\mu({\rm d}y_j)=0$ and there exists a constant $C_{\mu}\geq 0$ such that $\int_{-1/2}^{1/2}|y_j|^k\,\mu({\rm d}y_j)\leq C_\mu$ for all $k\geq 2$.
\end{enumerate}

\emph{Remark.} Certain holomorphic functions admit regularity bounds of the form stated in Assumption~\ref{assumpA2} (cf., e.g., \cite[Proposition~2.3]{NN1}). If $g$ is a holomorphic parametric map, then the $\bsb$ sequence controls the radii of the domains of analytic continuation, with $p$ related to the rate of decay of the $\bsnu$-th partial derivative of $g$ (see \cite{NN1} and references therein). The smaller the value of $p$, the faster the decay rate, which will be reflected in the dimension truncation error rates in Theorem~\ref{thm:1} and Corollary~\ref{cor:1}. Especially, solutions to elliptic PDEs with random diffusion coefficients fall into this framework.

If Assumption~\ref{assumpA2} holds, then we infer that $\bsy \mapsto G(g(\bsy))$ for all $G\in X'$ is continuous as a composition of continuous mappings. Hence $\bsy \mapsto G(g(\bsy))$ is measurable for all $G \in X'$, i.e., $\bsy \mapsto g(\bsy)$ is weakly measurable. Since $X$ is assumed to be a separable Hilbert space, by Pettis' theorem (cf., e.g., \cite[Chapter 4]{Yosida}) we obtain that $\bsy \mapsto g(\bsy)$ is strongly measurable. The upper bound in Assumption~\ref{assumpA2} is $\bsmu$-integrable. Thus we conclude from Bochner's theorem (cf., e.g., \cite[Chapter 5]{Yosida}) and Assumption~\ref{assumpA2} that $g$ is $\bsmu$-integrable over $U$.

Furthermore, $\bsmu$-a.e.~equality defines an equivalence relation among strongly $\bsmu$-measurable functions. 
By $L^2_{\bsmu}(U;X)$ we denote the Hilbert space of equivalence classes of strongly $\bsmu$-measurable functions $f:U\to X$ with norm
\begin{align*}
    \|f\|_{L^2_{\bsmu}(U;X)} := \bigg(\int_U \|f(\bsy)\|_X^2\,\bsmu(\mathrm d\bsy)\bigg)^{\frac12} <\infty. 
\end{align*}
Moreover, under the Assumptions~\ref{assumpA1} and~\ref{assumpA2} it can be shown that $g,g_s \in L^2_{\bsmu}(U;X)$ and $$
\lim_{s\to \infty} \|g(\bsy) - g(\bsy_{\leq s},\boldsymbol{0})\|_{L_{\bsmu}^2(U;X)} = \lim_{s\to \infty} \bigg(\int_{U} \|g(\bsy) - g(\bsy_{\leq s},\boldsymbol{0})\|_X^2\,\bsmu(\mathrm d\bsy)\bigg)^{\frac12} = 0,
$$
by applying Lebesgue's dominated convergence theorem (see, e.g., \cite[Theorem 1]{infiniteintegration} and \cite[Section 26]{Halmos}) to 
$$F^s(\bsy) := \|g(\bsy) - g(\bsy_{\leq s},\boldsymbol{0})\|_X^2,$$
which converges $\bsmu$-a.e.~to zero by Assumption~\ref{assumpA1}, and can be bounded by $(2\Theta_{0})^2$ by Assumption~\ref{assumpA2}. We use the superscript to avoid confusion with the notation used to denote dimensionally-truncated functions elsewhere in the document.

\section{Dimension truncation error}
We will require the following parametric regularity bound for the main dimension truncation result.
\begin{lemma}\label{lemma:regularity}
Under Assumption~\ref{assumpA2}, there holds
$$
|\partial^{\bsnu}\|g(\bsy)-g_s(\bsy)\|_X^2|\leq \bigg(\max_{0\leq \ell\leq |\bsnu|}\frac{2\Theta_\ell}{\ell!}\bigg)^2 (|\bsnu|+1)!\bsb^{\bsnu}\quad\text{for all}~\bsnu\in\mathcal F_k~\text{and}~\bsy\in U.
$$
\end{lemma}
\begin{proof}Let $\bsnu\in\mathcal F_k$. We apply the Leibniz product rule with respect to the inner product of the Hilbert space $X$ to obtain
\begin{align*}
\partial^{\bsnu}\|g(\bsy)-g_s(\bsy)\|_X^2&=\partial^{\bsnu}\langle g(\bsy)-g_s(\bsy),g(\bsy)-g_s(\bsy)\rangle_X\\
&=\sum_{\bsm\leq\bsnu}\binom{\bsnu}{\bsm}\langle \partial^{\bsm}(g(\bsy)-g_s(\bsy)),\partial^{\bsnu-\bsm}(g(\bsy)-g_s(\bsy))\rangle_X.
\end{align*}
Using the Cauchy--Schwarz inequality together with Assumption~\ref{assumpA2} yields
\begin{align*}
|\partial^{\bsnu}\|g(\bsy)-g_s(\bsy)\|_X^2|&\leq \sum_{\bsm\leq\bsnu}\binom{\bsnu}{\bsm}\|\partial^{\bsm}(g(\bsy)-g_s(\bsy))\|_X\|\partial^{\bsnu-\bsm}(g(\bsy)-g_s(\bsy))\|_X\\
&\leq 4 \sum_{\bsm\leq\bsnu}\binom{\bsnu}{\bsm}\Theta_{|\bsm|}\bsb^{\bsm}\Theta_{|\bsnu|-|\bsm|}\bsb^{\bsnu-\bsm}\\
&=4\bsb^{\bsnu}\sum_{\ell=0}^{|\bsnu|}\Theta_{\ell}\Theta_{|\bsnu|-\ell}\sum_{\substack{|\bsm|=\ell\\ \bsm\leq \bsnu}}\binom{\bsnu}{\bsm}\\
&=4\bsb^{\bsnu}\sum_{\ell=0}^{|\bsnu|}\Theta_{\ell}\Theta_{|\bsnu|-\ell}\frac{|\bsnu|!}{\ell!(|\bsnu|-\ell)!}\\
&\leq 4\bigg(\max_{0\leq \ell\leq |\bsnu|}\frac{\Theta_\ell}{\ell!}\bigg)^2 (|\bsnu|+1)!\bsb^{\bsnu},
\end{align*}
where we used the Vandermonde convolution $\sum_{\substack{|\bsm|=\ell\\ \bsm\leq\bsnu}}\binom{\bsnu}{\bsm}=\binom{|\bsnu|}{\ell}=\frac{|\bsnu|!}{\ell!(|\bsnu|-\ell)!}$.
\end{proof}
The main result of this document is stated below.
\begin{theorem}\label{thm:1}
Let $g(\bsy)\in X$, $\bsy\in U$, satisfy Assumptions~\ref{assumpA1}--\ref{assumpA3}. Then
$$
\|g-g_s\|_{L_{\boldsymbol\mu}^2(U;X)}=\mathcal O(s^{-\frac{1}{p}+\frac12}),
$$
where the implied coefficient is independent of $s$.
\end{theorem}
\begin{proof}
Let $s\geq 1$ and define
$$
F^{{s}}(\bsy):=\|g(\bsy)-g_s(\bsy)\|_X^2\quad\text{for}~\bsy\in U.
$$
In the special case of the uniform distribution $\bsmu({\rm d}\bsy)={\rm d}\bsy$, we can apply~\cite[Theorem~4.2]{gk22} to obtain
$$
\|g-g_s\|_{L^2(U;X)}^2=\bigg|\int_U(F^s(\bsy)-F^s(\bsy_{\leq s},\mathbf 0))\,{\rm d}\bsy\bigg|=\mathcal O(s^{-\frac{2}{p}+1}),
$$
from which the claim follows. For completeness, we present the proof below for the probability measure $\bsmu$ and because parts of the argument will also be useful to establish the invariance of the dimension truncation rate in Section~4.

By developing the Taylor expansion of $F^s$ about $(\bsy_{\leq s},\mathbf 0)$ with integral remainder and observing that $F^s(\bsy_{\leq s},\mathbf 0)=0$, we obtain 
\begin{align}
\begin{split}
F^s(\bsy)&=\sum_{\ell=1}^k \sum_{\substack{|\bsnu|=\ell\\ \nu_j=0~\forall j\leq s}}\frac{\bsy^{\bsnu}}{\bsnu!}\partial^{\bsnu}F^s(\bsy_{\leq s},\mathbf 0)\\
&\quad + \sum_{\substack{|\bsnu|=k+1\\ \nu_j=0~\forall j\leq s}}\frac{k+1}{\bsnu!}\bsy^{\bsnu}\int_0^1 (1-t)^k \partial^{\bsnu}F^s(\bsy_{\leq s},t\bsy_{>s})\,{\rm d}t,
\end{split}\label{eq:taylorexp}
\end{align}
where $\bsy_{>s}:=(y_j)_{j>s}$. Integrating both sides over $\bsy\in U$ yields
\begin{align*}
\int_U F^s(\bsy)\,\bsmu({\rm d}\bsy)&=\sum_{\ell=1}^k \sum_{\substack{|\bsnu|=\ell\\ \nu_j=0~\forall j\leq s}}\frac{1}{\bsnu!}\int_U \bsy^{\bsnu}\partial^{\bsnu}F^s(\bsy_{\leq s},\mathbf 0)\,\bsmu({\rm d}\bsy)\\
&\quad +\sum_{\substack{|\bsnu|=k+1\\ \nu_j=0~\forall j\leq s}}\frac{k+1}{\bsnu!}\int_U \int_0^1 (1-t)^k \bsy^{\bsnu}\partial^{\bsnu}F^s(\bsy_{\leq s},t\bsy_{>s})\,{\rm d}t\,\bsmu({\rm d}\bsy).
\end{align*}
If $\bsnu\in\mathcal F_k$ is such that $\nu_j=1$ for any $j>s$, then Fubini's theorem together with Assumption~\ref{assumpA3} imply for the summands appearing in the first term that
$$
\int_U \bsy^{\bsnu}\partial^{\bsnu}F^s(\bsy_{\leq s},\mathbf 0)\,\bsmu({\rm d}\bsy)=\underset{=0}{\underbrace{\bigg(\prod_{j>s}\int_{-\frac12}^{\frac12} y_j^{\nu_j}\,\mu({\rm d}y_j)\bigg)}}\int_{[-\frac12,\frac12]^s}\partial^{\bsnu}F^s(\bsy_{\leq s},\mathbf 0)\,\bsmu({\rm d}\bsy_{>s}).
$$
Therefore all multi-indices with any component equal to 1 can be removed from the first sum (especially, we can omit all multi-indices with $|\bsnu|=1$). Further, applying the regularity bound proved in Lemma~\ref{lemma:regularity} and writing open the definition of $F^s$ yields
\begin{align}
\begin{split}
\int_U \|g(\bsy)-g_s(\bsy)\|_X^2\,\bsmu({\rm d}\bsy)&\leq C_\mu^k\bigg(\max_{0\leq\ell\leq k}\frac{2\Theta_\ell}{\ell!}\bigg)^2(k+1)!\sum_{\ell=2}^k \sum_{\substack{|\bsnu|=\ell\\ \nu_j=0~\forall j\leq s\\ \nu_j\neq 1~\forall j>s}}\bsb^{\bsnu}\\
&\quad +C_\mu^{k+1}\bigg(\max_{0\leq\ell\leq k+1}\frac{2\Theta_\ell}{\ell!}\bigg)^2(k+2)!\sum_{\substack{|\bsnu|=k+1\\ \nu_j=0~\forall j\leq s}}\frac{1}{\bsnu!}\bsb^{\bsnu},
\end{split}\label{eq:splitsum}
\end{align}
where we used $\int_0^1 (1-t)^k\,{\rm d}t=\frac{1}{k+1}$ and Assumption~\ref{assumpA3}. The second term in~\eqref{eq:splitsum} can be estimated from above using the multinomial theorem in conjunction with Stechkin's lemma (cf., e.g., \cite[Lemma 3.3]{KressnerTobler}):
$$
\sum_{\substack{|\boldsymbol \nu|=k+1\\ \nu_j=0~\forall j\leq s}}\frac{1}{\bsnu!}{\bsb}^{\boldsymbol \nu}\leq \sum_{\substack{|\boldsymbol \nu|=k+1\\ \nu_j=0~\forall j\leq s}}\frac{|\boldsymbol \nu|!}{\boldsymbol \nu!}{\bsb}^{\boldsymbol \nu}=\bigg(\sum_{j>s}b_j\bigg)^{k+1}\leq s^{(k+1)(-\frac{1}{p}+1)}\bigg(\sum_{j\geq 1}{b}_j^p\bigg)^{\frac{k+1}{p}}.
$$
On the other hand, the first term in~\eqref{eq:splitsum} can be estimated similarly to~\cite{gantner}:
\begin{align*}
    &\sum_{\substack{2\leq |\boldsymbol \nu|\leq k\\ \nu_j=0~\forall j\leq s\\ \nu_j\neq 1~\forall j> s}}{\bsb}^{\boldsymbol \nu} \leq \sum_{\substack{0\neq |\boldsymbol{\nu}|_{\infty}\leq k\\ \nu_j=0~\forall j\leq s\\ \nu_j\neq 1~\forall j> s}}{\bsb}^{\boldsymbol \nu} 
    = -1 + \prod_{j>s} \bigg(1+ \sum_{\ell = 2}^{k} {b}_j^\ell\bigg) 
    = -1 + \prod_{j>s} \bigg(1+ {b}_j^2 \sum_{\ell = 0}^{k-2} {b}_j^\ell\bigg) \\
    &\leq -1 + \prod_{j>s} \bigg(1+ {b}_j^2 \underbrace{\sum_{\ell = 0}^{k-2} {b}_1^\ell}_{=: \beta_k}\bigg) 
    \leq  -1 + \exp{\Big(\beta_k \sum_{j> s} {b}_j^2\Big)} = \sum_{\ell \geq 1} \frac{1}{\ell!}\Big(\beta_k \sum_{j> s} {b}_j^2\Big)^{\ell}.
\end{align*}
Using $\sum_{j>s} b_j^2 \leq s^{-\frac{2}{p}+1} (\sum_{j\geq 1}{b}_j^p)^{\frac{2}{p}}$, which follows from  Stechkin's lemma, we further estimate
\begin{align*}
    \sum_{\ell \geq 1} \frac{1}{\ell!}\Big(\beta_k \sum_{j> s} {b}_j^2\Big)^{\ell} \leq s^{-\frac{2}{p}+1} \sum_{\ell \geq 1} \frac{1}{\ell!} (\beta_k \|{\bsb}\|_p^{2})^\ell = s^{-\frac{2}{p}+1} (-1 + \exp(\beta_k \|{\bsb}\|_p^2)
\end{align*}
since $s^{-\frac{2}{p}+1} \geq (s^{-\frac{2}{p}+1})^{\ell}$ for all $\ell \geq 1$.

Altogether, the above discussion yields the bound
$$
\|g(\bsy)-g_s(\bsy)\|_{L^2_{\bsmu}(U;X)}^2=\int_U \|g(\bsy)-g_s(\bsy)\|_X^2\,\bsmu({\rm d}\bsy)=\mathcal O(s^{-\frac{2}{p}+1}+s^{(k+1)(-\frac{1}{p}+1)}),
$$
where the implied coefficient is independent of $s$. Since we assumed that $k=\lceil \frac{1}{1-p}\rceil$, the assertion follows by taking the square root on both sides.
\end{proof}

\section{Invariance of the dimension truncation rate under transformations of variables}
An interesting consequence of the Taylor series argument used in Theorem~\ref{thm:1} is that the \emph{dimension truncation rate remains invariant under certain transformations of the variables}. This has been previously observed in the context of dimension truncation for integration problems under the periodic model~\cite{hhkks22}. To make this notion precise, let us consider a mapping $\bsxi\!:U\to U$, $\bsxi(\bsy):=(\xi(y_1),\xi(y_2),\ldots)$, which satisfies the following conditions:
\begin{enumerate}\setcounter{enumi}{3}
    \item\label{assumpA4} There hold $\xi(0)=0$ and $\int_{-1/2}^{1/2} \xi(y)\,\mu({{\rm d}y)}=0$.
    \item\label{assumpA5} There exists $C_\xi\geq 0$ such that $\int_{-1/2}^{1/2} |\xi(y)|^k\,\mu({\rm d}y)\leq C_\xi$ for all $k\geq 2$.
\end{enumerate}
Then we obtain the following as a consequence of Theorem~\ref{thm:1}.
\begin{corollary}\label{cor:1} Let $g(\bsy)\in X$, $\bsy\in U$, satisfy Assumptions~\ref{assumpA1}--\ref{assumpA3} and let $\bsxi\!:U\to U$ satisfy Assumptions~\ref{assumpA4}--\ref{assumpA5}. Define the $\bsxi$-transformed function $g_{\bsxi}$ by
$$
g_{\bsxi}(\bsy):=g(\bsxi(\bsy)),\quad \bsy\in U,
$$
and its dimension truncation by $g_{\bsxi,s}(\bsy):=g_\bsxi(\bsy_{\leq s},\mathbf 0)$ for $\bsy\in U$. Then
$$
\|g_{\bsxi}-g_{\bsxi,s}\|_{L_{\boldsymbol\mu}^2(U;X)}=\mathcal O(s^{-\frac{1}{p}+\frac12}),
$$
where the implied coefficient is independent of $s$.
\end{corollary}
\begin{proof}We introduce $F_{\boldsymbol\xi}^s(\bsy):=\|g_{\boldsymbol\xi}(\bsy)-g_{\boldsymbol\xi,s}(\bsy)\|_X^2$ for $\bsy\in U$. By carrying out the change of variables $\bsy\leftarrow\boldsymbol\xi(\bsy)$ in~\eqref{eq:taylorexp}, we obtain
\begin{align*}
F_{\bsxi}^s(\bsy)&=\sum_{\ell=1}^k \sum_{\substack{|\bsnu|=\ell\\ \nu_j=0~\forall j\leq s}}\frac{\bsxi(\bsy)^{\bsnu}}{\bsnu!}\partial^{\bsnu}F^s(\bsxi(\bsy_{\leq s},\mathbf 0))\\
&\quad + \sum_{\substack{|\bsnu|=k+1\\ \nu_j=0~\forall j\leq s}}\frac{k+1}{\bsnu!}\bsxi(\bsy)^{\bsnu}\int_0^1(1-t)^k \partial^{\bsnu}F^s(\bsxi(\bsy_{\leq s},t\bsy_{>s}))\,{\rm d}t.
\end{align*}
Integrating the above formula on both sides over $\bsy\in U$ and utilizing Lemma~\ref{lemma:regularity} as well as Assumptions~\ref{assumpA4}--\ref{assumpA5}, we obtain---in complete analogy with the proof of Theorem~\ref{thm:1}---that
\begin{align*}
\int_U \|g_{\bsxi}(\bsy)-g_{\bsxi,s}(\bsy)\|_X^2\,\bsmu({\rm d}\bsy)&\leq C_\xi^k\bigg(\max_{0\leq\ell\leq k}\frac{2\Theta_\ell}{\ell!}\bigg)^2(k+1)!\sum_{\ell=2}^k \sum_{\substack{|\bsnu|=\ell\\ \nu_j=0~\forall j\leq s\\ \nu_j\neq 1~\forall j>s}}\bsb^{\bsnu}\\
&\quad + C_\xi^{k+1}\bigg(\max_{0\leq\ell\leq k+1}\frac{2\Theta_\ell}{\ell!}\bigg)^2(k+2)!\sum_{\substack{|\bsnu|=k+1\\ \nu_j=0~\forall j\leq s}}\frac{1}{\bsnu!}\bsb^{\bsnu}.
\end{align*}
The desired result follows by exactly the same argument as in the proof of Theorem~\ref{thm:1}.
\end{proof}

As an application, with $U:=[-\tfrac12,\tfrac12]^{\mathbb N}$, let $\bsxi\!:U\to U$ satisfy the Assumptions~\ref{assumpA4} and~\ref{assumpA5}, let $D\subset \mathbb R^d$, $d\in\{1,2,3\}$, be a bounded Lipschitz domain, and let $f\!:D\to \mathbb R$ be a fixed source term. Consider the parametric PDE problem
\begin{align}
\begin{cases}
-\nabla\cdot(a_{\bsxi}(\bsx,\bsy)\nabla u_{\bsxi}(\bsx,\bsy))=f(\bsx),&\bsx\in D,~\bsy\in U,\\
u_{\bsxi}(\bsx,\bsy)=0,&\bsx\in\partial D,~\bsy\in U,
\end{cases}\label{eq:xipde}
\end{align}
endowed with the $\bsxi$-transformed diffusion coefficient
$$
a_{\bsxi}(\bsx,\bsy):=a_0(\bsx)+\sum_{j\geq 1} \xi(y_j)\psi_j(\bsx),\quad \bsx\in D,~\bsy\in U,
$$
which is assumed to satisfy the following:
\begin{enumerate}
    \item[6.] There exist $a_{\min},a_{\max}>0$ such that $0<a_{\min}\leq a_{\bsxi}(\bsx,\bsy)\leq a_{\max}<\infty$ for all $\bsx\in D$ and $\bsy\in U$.
    \item[7.] $a_0\in L^\infty(D)$ and $\psi_j\in L^\infty(D)$ for all $j\in\mathbb N$.
    \item[8.] $\sum_{j\geq 1} \|\psi_j\|_{L^\infty(D)}^p<\infty$ for some $p\in(0,1)$.
\end{enumerate}

In this case, the transformation $\bsxi(\bsy):=(\frac{1}{\sqrt 6}\sin(2\pi y_j))_{j\geq 1}$ corresponds to the so-called \emph{periodic model} studied in~\cite{hhkks22,kkkns22,kks20} when we define $\bsmu({\rm d}\bsy)={\rm d}\bsy$, i.e., the uniform probability measure. It is not difficult to see that $u_{\bsxi}$ is related to the solution of the problem
$$
\begin{cases}
-\nabla\cdot (a(\bsx,\bsy)\nabla u(\bsx,\bsy))=f(\bsx),&\bsx\in D,~\bsy\in U,\\
u(\bsx,\bsy)=0,&\bsx\in\partial D,~\bsy\in U,
\end{cases}
$$
subject to
$$
a(\bsx,\bsy)=a_0(\bsx)+\sum_{j=1}^\infty y_j\psi_j(\bsx),\quad \bsx\in D,~\bsy\in U,
$$
via the transformation $u_\bsxi(\bsx,\bsy)=u(\bsx,\bsxi(\bsy))$. Let $X:=H_0^1(D)$ be equipped with the norm $\|v\|_{X}:=\int_D \|\nabla v(\bsx)\|_{\mathbb R^d}^2\,{\rm d}\bsx$. The mapping $\bsy\mapsto u(\cdot,\bsy)\in X$ satisfies Assumptions~\ref{assumpA1}--\ref{assumpA3}: especially, there holds
\begin{align}
\|\partial^{\bsnu}u(\cdot,\bsy)\|_{X}\leq \frac{\|f\|_{X'}}{a_{\min}}|\bsnu|!\bsb^{\bsnu}\quad\text{for all}~\bsy\in U,~\bsnu\in\mathscr F,\label{eq:affineregularitybound}
\end{align}
where $\bsb:=(b_j)_{j\geq 1}$ is defined by setting $b_j:=\frac{\|\psi_j\|_{L^\infty(D)}}{ a_{\min}}$ for all $j\geq 1$ and $f\in X'=H^{-1}(D)$, the topological dual space of $X$. Meanwhile, the transformation $\bsxi$ satisfies Assumptions
\ref{assumpA4}--\ref{assumpA5}. Therefore % 
 Corollary~\ref{cor:1} can be used to deduce that
\begin{align}
\|u_{\bsxi}-u_{\bsxi,s}\|_{L^2_{\bsmu}(U;X)}=\mathcal O(s^{-\frac{1}{p}+\frac12}),\label{eq:dimtruncratepde}
\end{align}
where the constant is independent of the dimension $s$. The rate~\eqref{eq:dimtruncratepde} was obtained in~\cite{kkkns22} using a highly technical Neumann series approach, which is heavily dependent on the fact that the quantity of interest can be written as the solution to an affine parametric operator equation. Meanwhile, our approach does not require the problem to have an affine structure, since only some information about the behavior of the partial derivatives is needed. Especially in studies of PDE uncertainty quantification, such bounds are derived as a byproduct of quasi-Monte Carlo analysis, so the dimension truncation rate can be obtained ``for free'' without further analysis.

Moreover, if $X_h$ is a conforming finite element subspace of $X$, $u_{\bsxi,h}(\cdot,\bsy)\in X_h$ denotes the finite element discretization of $u_{\bsxi}(\cdot,\bsy)\in X$ for all $\bsy\in U$, and $u_{\bsxi,h,s}(\cdot,\bsy)\in X_h$ denotes the dimension truncation of $u_{\bsxi,h}(\cdot,\bsy)$ for all $\bsy\in U$, then the finite element solution $u_{\bsxi,h}$ satisfies the same parametric regularity bound as $u_{\bsxi}$, implying that
$$
\|u_{\bsxi,h}-u_{\bsxi,h,s}\|_{L^2_{\bsmu}(U;X)}=O(s^{-\frac{1}{p}+\frac12}),
$$
independently of $s$.

Finally, we present an example illustrating how our results can be applied to nonlinear quantities of interest.

\emph{Example.} Let $X:=H_0^1(D)$ as above. Consider the nonlinear quantity of interest
\begin{align}
G_{\mathrm{nl}}(v):=\|v\|_{X}^2:=\int_D \|\nabla v(\bsx)\|_{\mathbb R^d}^2\,{\rm d}\bsx,\quad v\in X.\label{eq:nonlinearqoi}
\end{align}
If $u(\cdot,\bsy)\in X$ is the solution to~\eqref{eq:xipde} with $U=[-\tfrac12,\tfrac12]^{\mathbb N}$, $\bsmu({\rm d}\bsy):={\rm d}\bsy$, and $\bsxi(\bsy):=\bsy$, then it is known to satisfy Assumptions~\ref{assumpA1}--\ref{assumpA3} with the regularity bound~\eqref{eq:affineregularitybound}. Letting $C:=\frac{\|f\|_{X'}}{a_{\min}}$, we obtain by the Leibniz product rule that
\begin{align*}
\partial^{\bsnu}G_{\mathrm{nl}}
(u(\cdot,\bsy))&=\int_D\sum_{\bsm\leq \bsnu} \binom{\bsnu}{\bsm} \nabla\partial^{\bsm}u(\bsx,\bsy)\cdot \nabla \partial^{\bsnu-\bsm}u(\bsx,\bsy)\,{\rm d}\bsx\\
&\leq \sum_{\bsm\leq \bsnu}\binom{\bsnu}{\bsm}\|\partial^{\bsm}u(\cdot,\bsy)\|_{X}\|\partial^{\bsnu-\bsm}u(\cdot,\bsy)\|_{X}\\
&\leq C^2\sum_{\bsm\leq \bsnu}\binom{\bsnu}{\bsm}|\bsm|!\bsb^{\bsm}|\bsnu-\bsm|!\bsb^{\bsnu-\bsm}\\
&=C^2\bsb^{\bsnu}\sum_{\ell=0}^{|\bsnu|}\ell!(|\bsnu|-\ell)!\sum_{\substack{\bsm\leq \bsnu\\|\bsm|=\ell}}\binom{\bsnu}{\bsm}\\
&=C^2\bsb^{\bsnu}(|\bsnu|+1)!,
\end{align*}
where we used the Vandermonde convolution $\sum_{\substack{|\bsm|=\ell\\ \bsm\leq \bsnu}}\binom{\bsnu}{\bsm}=\binom{|\bsnu|}{\ell}=\frac{|\bsnu|!}{\ell!(|\bsnu|-\ell)!}$.

It follows from Theorem~\ref{thm:1} that
$$
\|G_{\mathrm{nl}}(u)-G_{\mathrm{nl}}(u_s)\|_{L_{\bsmu}^2(U;\mathbb R)}=O(s^{-\frac{1}{p}+\frac12}),
$$
independently of $s$. Moreover, for any $\bsxi\!:U\to U$ satisfying Assumptions \ref{assumpA4}--\ref{assumpA5}, it follows from Corollary~\ref{cor:1} that
$$
\|G_{\mathrm{nl}}(u_{\bsxi})-G_{\mathrm{nl}}(u_{\bsxi,s})\|_{L_{\bsmu}^2(U;\mathbb R)}=O(s^{-\frac{1}{p}+\frac12}),
$$
independently of $s$. We note that this especially holds for the periodic transformation $\bsxi(\bsy):=(\frac{1}{\sqrt 6}\sin(2\pi y_j))_{j\geq 1}$.

\section{Numerical experiments}
Let $D=(0,1)^2$ be a spatial domain, $U=[-\tfrac12,\tfrac12]^{\mathbb N}$, and $f(\bsx):=x_1$ a fixed source term. We let $\bsxi\!:U\to U$, $\bsxi(\bsy)=(\frac{1}{\sqrt 6}\sin(2\pi y_j))_{j\geq 1}$, and $\bsmu({\rm d}\bsy)={\rm d}\bsy$ denotes the uniform probability measure. We consider the PDE problem
\begin{align}
\begin{cases}
-\nabla\cdot(a_{\bsxi}(\bsx,\bsy)\nabla u_{\bsxi}(\bsx,\bsy))=f(\bsx),&\bsx\in D,~\bsy\in U,\\
u_{\bsxi}(\bsx,\bsy)=0,&\bsx\in\partial D,~\bsy\in U,
\end{cases}\label{eq:pdenumex}
\end{align}
equipped with the diffusion coefficient
\begin{align*}
a_{\bsxi}(\bsx,\bsy)=\frac32+\sum_{j\geq 1}\xi(y_j)j^{-\vartheta}\sin(j\pi x_1)\sin(j\pi x_2),\quad \bsx\in D,~\bsy\in U,~\vartheta>1.%\label{eq:pdenumex2}
\end{align*}
The PDE~\eqref{eq:pdenumex} is spatially discretized using a first-order conforming finite element method with mesh size $h=2^{-5}$.

We consider the dimension truncation error for the full PDE solution using the formula
\begin{align*}
\|u_{\bsxi}-u_{\bsxi,s}\|_{L^2(U;L^2(D))}\approx\bigg(\int_{[-\frac12,\frac12]^{s'}}\|u_{\bsxi,s'}(\cdot,\bsy)-u_{\bsxi,s}(\cdot,\bsy)\|_{L^2(D)}^2\,{\rm d}\bsy\bigg)^{\frac12},%
\end{align*}
and we also consider the nonlinear quantity of interest~\eqref{eq:nonlinearqoi}, estimating the dimension truncation error using the formula
\begin{align*}
\|G_{\rm nl}(u_\bsxi)-G_{\rm nl}(u_{\bsxi,s})\|_{L^2(U)}\approx \bigg(\int_{[-\frac12,\frac12]^{s'}}|G_{\rm nl}(u_{\bsxi,s'}(\cdot,\bsy))-G_{\rm nl}(u_{\bsxi,s}(\cdot,\bsy))|^2\,{\rm d}\bsy\bigg)^{\frac12}.%
\end{align*}
In both cases, we choose $s'\gg s$ and the high-dimensional integrals are approximated using a randomly shifted rank-1 lattice rule with $2^{20}$ cubature nodes and a single random shift. As the integration lattice, we use in both cases an off-the-shelf rank-1 lattice rule~\cite[lattice-39101-1024-1048576.3600]{KuoLattice} and use the same random shift for each value of $\vartheta$. As the reference solution, we use the PDE solution corresponding to $s'=2^{11}$.

The numerical results for dimensions $s\in\{2^k:k=1,\ldots,9\}$ and decay rates $\vartheta\in\{1.5,2.0,3.0\}$ corresponding to the full PDE solution and the nonlinear quantity of interest are displayed in Figures~\ref{fig:1} and~\ref{fig:2}, respectively. The theoretical convergence rates in each case are $-1.0$, $-1.5$, and $-2.5$, respectively, and they are displayed alongside the numerical results.

\begin{figure}[!t]
    \centering
    \includegraphics[width=.8\textwidth]{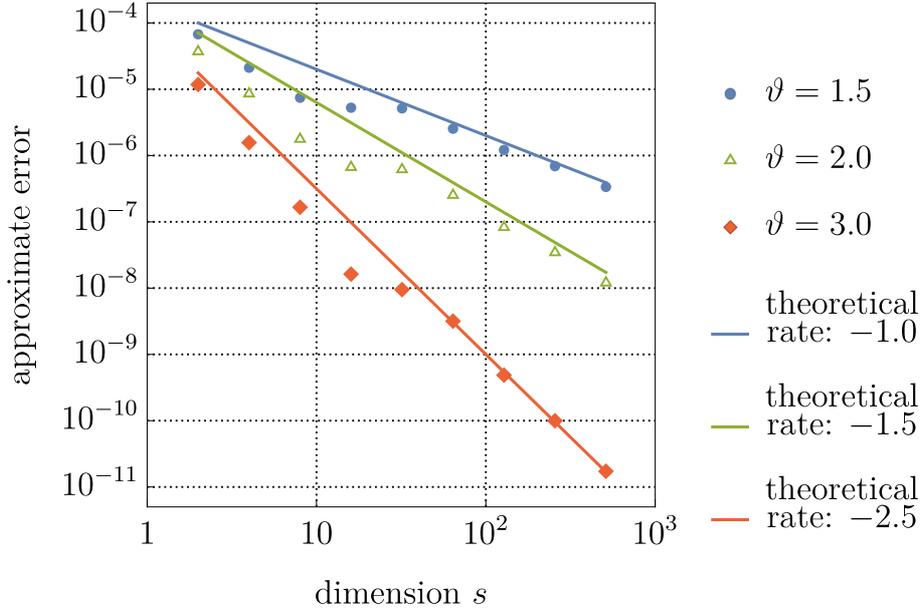}
    \caption{The dimension truncation errors of the full PDE solution corresponding to a periodically parameterized input random field with decay parameters $\vartheta\in\{1.5,2.0,3.0\}$. The expected dimension truncation error rates are $-1.0$, $-1.5$, and $-2.5$, respectively.}
    \label{fig:1}
\end{figure}

\begin{figure}
    \centering
    \includegraphics[width=.8\textwidth]{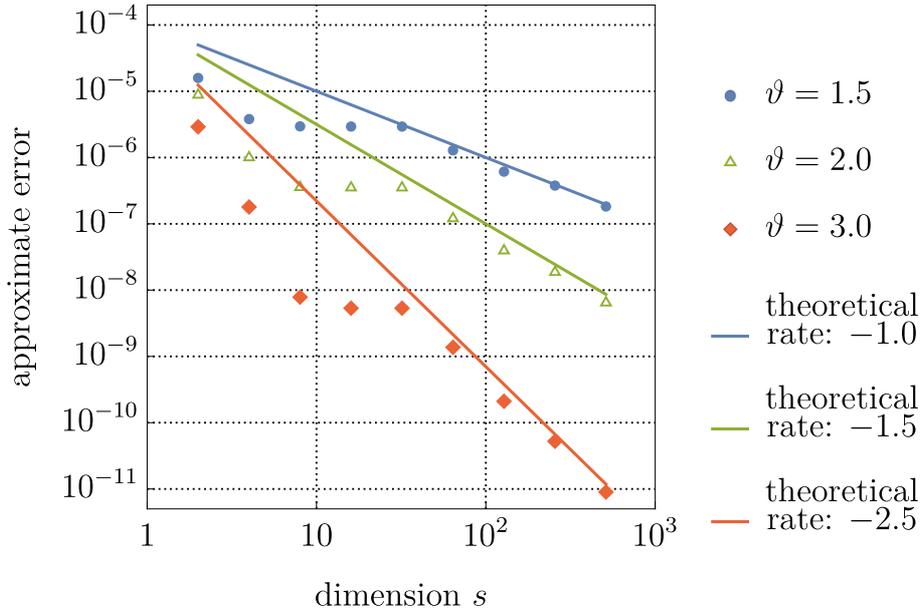}
    \caption{The dimension truncation errors of the nonlinear quantity of interest corresponding to a periodically parameterized input random field with decay parameters $\vartheta\in\{1.5,2.0,3.0\}$. The expected dimension truncation error rates are $-1.0$, $-1.5$, and $-2.5$, respectively.}
    \label{fig:2}
\end{figure}

The convergence graphs corresponding to the full PDE solution in Figure~\ref{fig:1} plateau between $10\leq s\leq 100$, which may be explained by the contributions of the finite element discretization error as well as the use of an ``off-the-shelf'' lattice rule (in contrast to a ``tailored'' lattice rule). This behavior appears to be exacerbated in the convergence graphs corresponding to the nonlinear quantity of interest in Figure~\ref{fig:2}. Nonetheless, in all cases the theoretically anticipated convergence rates are easily observed in practice. We remark that the convergence graphs corresponding to the affine and uniform model with $\boldsymbol\xi(\bsy):=(y_j)_{j\geq 1}$ are extremely similar to the results corresponding to the periodic model, and have thus been omitted.

\section{Conclusions}
Unlike many studies which have considered the dimension truncation error rate within the context of high-dimensional numerical integration, we considered the $L^2$ dimension truncation error rate for parametric Hilbert space valued functions. Our theory covers both affine parametric as well as non-affine parametric problems with sufficiently smooth dependence on a sequence of bounded, parametric variables. The main dimension truncation results presented in this work can be applied to nonlinear quantities of interest of parametric model problems, provided that they satisfy the conditions of our framework. In addition, the Hilbert space can be chosen to be a finite element subspace, indicating that our dimension truncation results are also valid for conforming finite element approximations of parametric PDEs.

The $L^2$ dimension truncation error rates considered in this work arise, e.g., in the study of high-dimensional function approximation of parametric PDEs. An example of such an approximation scheme is the kernel method over lattice point sets considered in~\cite{kkkns22}. The kernel method was analyzed in the context of the so-called \emph{periodic model}, in which a countable number of independent random variables enter the input random field of the PDE as periodic functions. Our second main result shows that the $L^2$ dimension truncation error rate remains invariant under certain transformations of the parametric variables: especially, the $L^2$ dimension truncation rate considered in this work holds for periodically parametrized model problems such as those studied in~\cite{hhkks22,kkkns22,kks20}.

Some potential extensions of this work include analysis in $L^q$ spaces for $1\leq q\leq \infty$, which would require a modification of Lemma~\ref{lemma:regularity}, as well as unbounded parameter domains such as $U=\mathbb R^{\mathbb N}$, which would require an extension of the proof developed in~\cite{gk22} to the function approximation setting.

%

%%%%%%%%%%%%%%%%%%%%%%%% referenc.tex %%%%%%%%%%%%%%%%%%%%%%%%%%%%%%
% sample references
% %
% Use this file as a template for your own input.
%
%%%%%%%%%%%%%%%%%%%%%%%% Springer-Verlag %%%%%%%%%%%%%%%%%%%%%%%%%%
%
% BibTeX users please use
% \bibliographystyle{}
% \bibliography{}
%

\end{document}